\documentclass[a4paper]{amsart}

\usepackage{amsfonts, latexsym, amssymb, graphicx,
multicol, mathrsfs, color, amscd, verbatim, paralist,
xspace, url, euscript, stmaryrd,  amsmath, amscd, enumitem,
bbold, multirow, tikz, mathtools, mathabx}
\usepackage[all,pdf,cmtip]{xy}

\usepackage[colorlinks, linkcolor=blue, citecolor=magenta, urlcolor=cyan]{hyperref}

\usepackage{mathabx}

\def\ra{\rightarrow}

\def\qed{{\hfill $\Box$}}
\newtheorem{theorem}{THEOREM}[section]
\newtheorem{corollary}[theorem]{Corollary}
\newtheorem{proposition}[theorem]{Proposition}
\newtheorem{lemma}[theorem]{Lemma}

\theoremstyle{definition}
\newtheorem{definition}[theorem]{Definition}
\theoremstyle{remark}
\newtheorem{remark}[theorem]{Remark}

\newcommand\CC{{\mathbb C}}
\newcommand\RR{{\mathbb R}}

\def\GL{\mathop{\rm GL}\nolimits}

\def\Re{\mathop{\rm Re}\nolimits}

\makeatletter
\def\blfootnote{\xdef\@thefnmark{}\@footnotetext}
\makeatother

\begin{document}

\title[Zero CR-curvature of Levi degenerate hypersurfaces]{Zero CR-curvature equations
\vspace{0.1cm}\\
for Levi degenerate hypersurfaces
\vspace{0.1cm}\\
via Pocchiola's invariants}
\author[Isaev]{Alexander Isaev}

\address{Mathematical Sciences Institute\\
Australian National University\\
Canberra, Acton, ACT 2601, Australia}
\email{alexander.isaev@anu.edu.au}

\maketitle

\thispagestyle{empty}

\pagestyle{myheadings}

\begin{abstract} 
In articles \cite{I2, I3} we studied tube hypersurfaces in $\CC^3$ that are 2-nondegenerate and uniformly Levi degenerate of rank 1. In particular, we showed that the vanishing of the CR-curvature of such a hypersurface is equivalent to the Monge equation with respect to one of the variables. In the present paper we provide an alternative shorter derivation of this equation by utilizing two invariants discovered by S. Pocchiola. We also investigate Pocchiola's invariants in the rigid case and give a partial classification of rigid 2-nondegenerate uniformly Levi degenerate of rank 1 hypersurfaces with vanishing CR-curvature. 
\end{abstract}

\section{Introduction}\label{intro}
\setcounter{equation}{0}

This paper is a continuation of articles \cite{I2, I3}, and we will extensively refer the reader to these papers in what follows. In particular, a brief review of CR-geometric concepts is given in \cite[Section 2]{I2}, and we will make use of those concepts without further reference.

We consider connected $C^{\infty}$-smooth real hypersurfaces in $\CC^n$ with $n\ge 2$. First, we look at \emph{tube hypersurfaces} or simply \emph{tubes}, that is, locally closed real submanifolds of the form
$$
M={\mathcal S}+i\RR^n,
$$
where ${\mathcal S}$ is a $C^{\infty}$-smooth locally closed hypersurface in $\RR^n\subset\CC^n$ called the \emph{base} of $M$. Two tube hypersurfaces are said to be \emph{affinely equivalent} if there is an affine transformation of $\CC^n$ of the form
\begin{equation}
z\mapsto Az+b,\quad A\in\GL_n(\RR),\quad b\in\CC^n,\label{affequiv}
\end{equation}
mapping one hypersurface onto the other (which takes place if and only if the bases of the tubes are affinely equivalent as submanifolds of $\RR^n$).

There has been a significant effort to study the interplay between the CR-geometric and affine-geometric properties of tubes (see \cite[Section 1]{I2} for a bibliography on this subject). Specifically, the following question has attracted substantial interest:
\vspace{-0.5cm}\\

$$
\begin{array}{l}
\hspace{0.2cm}\hbox{$(*)$ when does (local or global) CR-equivalence of tubes imply}\\
\vspace{-0.3cm}\hspace{0.8cm}\hbox{affine equivalence?}\\
\end{array}
$$
\vspace{-0.4cm}\\

\noindent Up until recently, an acceptable answer to the above question has only been known for Levi nondegenerate tube hypersurfaces that are in addition \emph{CR-flat}, that is, have identically vanishing CR-curvature (see \cite{I1} for an exposition of the existing theory). In an attempt to relax the Levi nondegeneracy assumption, in article \cite{I2} we initiated an investigation of question $(*)$ for a class of Levi degenerate 2-nondegenerate tube hypersurfaces while still assuming CR-flatness. As part of our considerations, we analyzed CR-curvature for this class.

We note that CR-curvature is defined in situations when the CR-structures in question are reducible to absolute parallelisms with values in a Lie algebra ${\mathfrak g}$. Indeed, suppose we have a class ${\mathfrak C}$ of CR-manifolds. Then the CR-structures in ${\mathfrak C}$ are said to \emph{reduce to ${\mathfrak g}$-valued absolute parallelisms} if to every $M\in{\mathfrak C}$ one can assign a fiber bundle ${\mathcal P}_M\ra M$ and an absolute parallelism $\omega_M$ on ${\mathcal P}_M$ such that for every $p\in {\mathcal P}_M$ the parallelism establishes an isomorphism between $T_p({\mathcal P}_M)$ and ${\mathfrak g}$ and for all $M_1,M_2\in{\mathfrak C}$ the following holds: 

\noindent (i) every CR-isomorphism $f:M_1\ra M_2$ can be lifted to a diffeomorphism\linebreak $F: {\mathcal P}_{M_{{}_1}}\ra{\mathcal P}_{M_{{}_2}}$ satisfying
\begin{equation}
F^{*}\omega_{M_{{}_2}}=\omega_{M_{{}_1}},\label{eq8}
\end{equation}
and 

\noindent (ii) any diffeomorphism $F: {\mathcal P}_{M_{{}_1}}\ra{\mathcal P}_{M_{{}_2}}$ satisfying (\ref{eq8}) 
is a bundle isomorphism that is a lift of a CR-isomorphism $f:M_1\ra M_2$. 

In this situation one considers the ${\mathfrak g}$-valued \emph{CR-curvature form}
$$
\Omega_M:=d\omega_M-\frac{1}{2}\left[\omega_M,\omega_M\right],\label{genformulacurvature}
$$
and the CR-flatness of $M$ is the condition of the identical vanishing of $\Omega_M$ on the bundle ${\mathcal P}_M$.
 
Reducing CR-structures to absolute parallelisms goes back to \'E. Cartan  who produced reduction for all 3-dimensional Levi nondegenerate CR-hyper\-surfaces (see \cite{C}). Since then there have been many developments under the assumption of Levi nondegeneracy (see \cite[Section 1]{I2} for references). On the other hand, reducing \emph{Levi degenerate} CR-structures has proved to be quite hard, and the first result for a reasonably large class of Levi degenerate manifolds only appeared in 2013 in our paper \cite{IZ}. Specifically, we looked at the class ${\mathfrak C}_{2,1}$ of connected 5-dimensional CR-hypersurfaces that are 2-nondegenerate and uniformly Levi degenerate of rank 1 and proved that the CR-structures in ${\mathfrak C}_{2,1}$ reduce to ${\mathfrak{so}}(3,2)$-valued parallelisms. Alternative constructions were presented in \cite{MS}, \cite{MP}, \cite{Poc} (see also \cite{Por}, \cite{PZ} for reduction in higher-dimensional cases). One of the results of \cite{IZ} is that a manifold $M\in{\mathfrak C}_{2,1}$ is CR-flat (with respect to our reduction) if and only if in a neighborhood of its every point $M$ is CR-equivalent to an open subset of the \emph{tube hypersurface over the future light cone in $\RR^3$}
$$
M_0:=\left\{(z_1,z_2,z_3)\in\CC^3\mid (\Re z_1)^2+(\Re z_2)^2-(\Re z_3)^2=0,\,\ \Re z_3>0\right\}.\label{light}
$$

Now, the main result of \cite{I2} (see \cite[Theorem 1.1]{I2}) asserts that every CR-flat tube hypersurface in ${\mathfrak C}_{2,1}$ is affinely equivalent to an open subset of $M_0$. This conclusion is a complete answer to question $(*)$ in the situation at hand and is in stark contrast to the Levi nondegenerate case where the CR-geometric and affine-geometric classifications are different even in low dimensions.

The key part of our arguments in \cite{I2} was to write the zero CR-curvature equations for tubes in the class ${\mathfrak C}_{2,1}$ in convenient form (see Theorem \ref{flatness} below). Interestingly, as we showed in the follow-up paper \cite{I3}, these equations are equivalent to a \emph{single} partial differential equation, which we will now describe. Let $M$ be a tube hypersurface in the class ${\mathfrak C}_{2,1}$. Up to affine equivalence, the base of $M$ is given locally by the graph of a function of two variables as in (\ref{basiceq}), where condition (\ref{rho11}) is satisfied. For this local representation the equation of zero CR-curvature for $M$ is the well-known \emph{Monge equation} with respect to the first variable shown in (\ref{veryfinalthetasss}). Recall that the classical single-variable Monge equation admits a nice geometric interpretation: it describes all planar conics (see, e.g., \cite[pp.~51--52]{Lan}, \cite{Las}). Furthermore, as explained in \cite{I2}, all solutions of (\ref{veryfinalthetasss}) can be explicitly found as well, and every solution yields a tube hypersurface affinely equivalent to an open subset of $M_0$.

Unfortunately, the calculations in \cite{I2} that lead to equation (\ref{veryfinalthetasss}) are quite hard and not very transparent as they rely on the construction of \cite{IZ}, which is rather involved. It is therefore desirable to find a more elementary way of obtaining this equation. The first result of the present paper is an easier derivation of (\ref{veryfinalthetasss}) by utilizing two CR-invariants, called $J$ and $W$, introduced by S. Pocchiola in \cite{Poc} for real hypersurfaces in $\CC^3$ in the class ${\mathfrak C}_{2,1}$. These invariants are given explicitly (albeit by very lengthy formulas) in terms of a graphing function, and CR-flatness is equivalent to their simultaneous vanishing:
\begin{equation}
\left\{\begin{array}{l}
J=0,\\
\vspace{-0.3cm}\\
W=0.
\end{array}\right.\label{condsjw}
\end{equation}
As shown in Section \ref{main}, equation (\ref{veryfinalthetasss}) is a relatively easy consequence of system (\ref{condsjw}) in the tube case. Moreover, the calculations of Section \ref{main} yield a short proof of the following fact, which was initially established in \cite{I2} by a much longer argument:

\begin{theorem}\label{flatness}
Let $M$ be a tube hypersurface in $\CC^3$ and assume that $M\in{\mathfrak C}_{2,1}$. Fix $x\in M$ and represent $M$ near $x$, up to affine equivalence, by {\rm (\ref{basiceq})}, where the function $\rho$ satisfies {\rm (\ref{rho11})}. Then the vanishing of the CR-curvature of this local representation is equivalent to the pair of equations
\begin{equation}
\left\{\begin{array}{l}
\hbox{the Monge equation w.r.t. the first variable {\rm (\ref{veryfinalthetasss})}},\\
\vspace{-0.1cm}\\
\displaystyle\frac{\partial S}{\partial t_1}=0,
\end{array}\right.\label{threeeeqs}
\end{equation}
where $S$ is the function defined in {\rm (\ref{functions})}. 
\end{theorem}

We note that the content of paper \cite{I3} was in showing that the second equation in system (\ref{threeeeqs}) is a consequence of the first one, so (\ref{threeeeqs}) in fact reduces to a single equation. Further, owing to Levi degeneracy, the graphing function of the base of $M$ satisfies the \emph{homogeneous Monge-Amp\`ere equation} (see (\ref{mongeampere})), which plays an important role in our analysis. Hence, the problem of locally determining all CR-flat tubes in the class ${\mathfrak C}_{2,1}$ is described by the following system of partial differential equations:
\begin{equation}
\left\{\begin{array}{l}
\hbox{the Monge equation w.r.t. first variable (\ref{veryfinalthetasss})},\\
\vspace{-0.3cm}\\
\hbox{the Monge-Amp\`ere equation (\ref{mongeampere})},
\end{array}\right.\label{twoeeqs}
\end{equation}
where, in addition to (\ref{rho11}), the function $S$ is required to be everywhere nonvanishing (see (\ref{snonzero})). In fact, as explained in Section \ref{conclus}, equation (\ref{veryfinalthetasss}) in some coordinates becomes the standard \emph{single-variable} Monge equation (\ref{classicalMongeeq}), so (\ref{twoeeqs}) turns into a remarkable system of two classical equations: 
$$
\left\{\begin{array}{l}
\hbox{the Monge equation (\ref{classicalMongeeq})},\\
\vspace{-0.3cm}\\
\hbox{the Monge-Amp\`ere equation (\ref{mongeampere})}.
\end{array}\right.
$$
Curiously, both equations in this system happen to be named after Gaspard Monge.

Next, Section \ref{rigid} concerns \emph{rigid hypersurfaces} in $\CC^3$, i.e., hypersurfaces defined by equations of the form (\ref{rigideq}). After the tube case, this is the next situation up in terms of complexity where the formulas for Pocchiola's invariants $J$ and $W$ still look manageable. Note, however, that the rigid case is much more complicated than the tube one even in the Levi nondegenerate CR-flat setup in $\CC^2$, for which a (somewhat) explicit classification has only recently been found (see \cite{ES} and the earlier work \cite{S}). For Levi degenerate rigid hypersurfaces in $\CC^3$ obtaining a reasonable description is even harder. Nevertheless, upon introducing additional assumptions, we manage to partially classify the germs of CR-flat rigid hypersurfaces in the class ${\mathfrak C}_{2,1}$ up to a certain natural equivalence, which we call \emph{rigid equivalence}. This classification is the second main result of the paper and appears in Theorem \ref{rigidclassif}. It would be interesting to explore whether one can relax the assumptions of the theorem while still being able to produce explicit equations. One consequence of Theorem \ref{rigidclassif} is a rationality result (see Corollary \ref{rational}). It would be rather intriguing to investigate whether rationality holds true regardless of the simplifying assumptions of Theorem \ref{rigidclassif} (see Remark \ref{remrational}). 

{\bf Acknowledgement.} This work was done while the author was visiting the Steklov Mathematical Institute in Moscow.

\section{Preliminaries}\label{prelim}
\setcounter{equation}{0}

Let $M$ be any tube hypersurface in $\CC^3$. For $x\in M$, a \emph{tube neighborhood} of $x$ in $M$ is an open subset $U$ of $M$ that contains $x$ and has the form $M\cap({\mathcal U}+i\RR^3)$, where ${\mathcal U}$ is an open subset of $\RR^3$. It is easy to see that for every point $x\in M$ there exists a tube neighborhood $U$ of $x$ in $M$ and an affine transformation of $\CC^3$ as in (\ref{affequiv}) that maps $x$ to the origin and establishes affine equivalence between $U$ and a tube hypersurface of the form
\begin{equation}
\begin{array}{l}
\Gamma_{\rho}:=\{(z_1,z_2,z_3): z_3+{\bar z}_3=\rho(z_1+{\bar z}_1,z_2+{\bar z}_2)\}=\\
\vspace{-0.1cm}\\
\hspace{4cm}\displaystyle\left\{(z_1,z_2,z_3): \Re z_3=\frac{1}{2}\,\rho(2\Re z_1,2\Re z_2)\right\},
\end{array}\label{basiceq}
\end{equation}
where $\rho(t_1,t_2)$ is a smooth function defined in a neighborhood of 0 in $\RR^2$ with
\begin{equation}
\rho(0)=0,\quad \rho_1(0)=0, \quad \rho_2(0)=0\label{initial}
\end{equation}
(from this moment until the end of Section \ref{conclus} subscripts 1 and 2 indicate partial derivatives with respect to $t_1$ and $t_2$). In what follows, we will be mainly interested in the germ of $\Gamma_{\rho}$ at the origin, so the domain of $\rho$ will be allowed to shrink if necessary. Equations of the form
\begin{equation}
\Re z_3=F(\Re z_1,\Re z_2),\label{tubeeq}
\end{equation}
where $F$ is a smooth real-valued function on a domain in $\RR^2$, are often called \emph{tube equations}. Every tube hypersurface is locally affinely equivalent to the hypersurface defined by a tube equation.  

Let now $M$ be uniformly Levi degenerate of rank 1. Then the Hessian matrix of $\rho$ has rank 1 at every point, hence $\rho$ is a solution of the homogeneous Monge-Amp\`ere equation
\begin{equation}
\rho_{11}\rho_{22}-(\rho_{12})^2=0,\label{mongeampere}
\end{equation}
where, by invoking affine equivalence, one can additionally assume
\begin{equation}
\hbox{$\rho_{11}>0$ everywhere.}\label{rho11}
\end{equation}

Next, in \cite[Section 3]{I2} we showed that for $\rho$ satisfying (\ref{mongeampere}), (\ref{rho11}), the hypersurface $\Gamma_{\rho}$ is 2-nondegenerate if and only if the function
\begin{equation}
S:=\left(\frac{\rho_{12}}{\rho_{11}}\right)_{1}\label{functions} 
\end{equation}
vanishes nowhere (cf. \cite{MP,Poc}). Thus, assuming that $M$ is 2-nondegenerate, we have
\begin{equation}
\hbox{$S\ne 0$ everywhere.}\label{snonzero}
\end{equation}

Further, consider the equation
\begin{equation}
9\rho^{{\rm(V)}}(\rho_{11})^{2}-45\rho^{{\rm(IV)}}\rho_{111}\rho_{11}+40(\rho_{111})^{3}=0, 
\label{veryfinalthetasss}
\end{equation}
where $\rho^{{\rm(IV)}}:=\partial^{\,4}\rho/\partial\,t_1^4$, $\rho^{{\rm(V)}}:=\partial^{\,5}\rho/\partial\,t_1^5$. This is the Monge equation with respect to the first variable that appears in system (\ref{twoeeqs}). In the next section we will see that (\ref{veryfinalthetasss}) follows from system (\ref{condsjw}). Thus, the detailed form of system (\ref{twoeeqs}) is
$$
\left\{\begin{array}{l}
9\rho^{{\rm(V)}}(\rho_{11})^{2}-45\rho^{{\rm(IV)}}\rho_{111}\rho_{11}+40(\rho_{111})^{3}=0,\\
\vspace{-0.3cm}\\
\rho_{11}\rho_{22}-(\rho_{12})^2=0,
\end{array}\right.\label{twoeeqs1}
$$
where conditions (\ref{rho11}) and (\ref{snonzero}) are satisfied.

We now turn to Pocchiola's invariants $J$ and $W$. Explicit  formulas for them in terms of a graphing function of a hypersurface in $\CC^3$ in the class ${\mathfrak C}_{2,1}$ are given in \cite{MP,Poc}. They are quite complicated in general, but for the hypersurface $\Gamma_{\rho}$ defined by (\ref{basiceq}) these formulas are easily seen to simplify as
\begin{equation}
\begin{array}{l}
\displaystyle J=\frac{5(S_1)^2}{18 S^2}\frac{\rho_{111}}{\rho_{11}}+\frac{1}{3}\frac{\rho_{111}}{\rho_{11}}\left(\frac{\rho_{111}}{\rho_{11}}\right)_1-\frac{S_1}{9 S}\frac{(\rho_{111})^2}{(\rho_{11})^2}+\frac{20(S_1)^3}{27 S^3}-\frac{5 S_1 S_{11}}{6 S^2}+\\
\vspace{-0.3cm}\\
\hspace{0.7cm}\displaystyle\frac{S_1}{6 S}\left(\frac{\rho_{111}}{\rho_{11}}\right)_1-\frac{S_{11}}{6 S}\frac{\rho_{111}}{\rho_{11}}-\frac{2}{27}\frac{(\rho_{111})^3}{(\rho_{11})^3}-\frac{1}{6}\left(\frac{\rho_{111}}{\rho_{11}}\right)_{11}+\frac{S_{111}}{S_1},\\
\vspace{-0.1cm}\\
\displaystyle W=\frac{4S_1}{3S}+\frac{S_1}{3S^3}\left(\frac{\rho_{12}}{\rho_{11}}S_1-S_2\right)-\frac{1}{3S^2}\left(\frac{\rho_{12}}{\rho_{11}}S_{11}-S_{12}\right),
\end{array}\label{expljw}
\end{equation}
where all functions of the variables $t_1$, $t_2$ are calculated for
\begin{equation}
\begin{array}{l}
t_1=z_1+{\bar z}_1,\\
\vspace{-0.3cm}\\
t_2=z_2+{\bar z}_2.
\end{array}\label{substitution} 
\end{equation}

We have now collected all the facts required for obtaining equation (\ref{veryfinalthetasss}) from system (\ref{condsjw}) and, moreover, for giving a short proof of Theorem \ref{flatness}.

\section{Proof of Theorem \ref{flatness}}\label{main}
\setcounter{equation}{0}

In our proof we extensively use the fact that the function $\rho$ satisfies the homogeneous Monge-Amp\`ere equation (\ref{mongeampere}), so we start by recalling classical facts concerning its solutions. For details the reader is referred to paper \cite{U}. 

Let us make the following change of variables near the origin:
\begin{equation}
\begin{array}{l}
v=\rho_1(t_1,t_2),\\
\vspace{-0.3cm}\\
w=t_2
\end{array}\label{changevar}
\end{equation}
and set
\begin{equation}
\begin{array}{l}
p(v,w):=\rho_2(t_1(v,w),w),\\
\vspace{-0.3cm}\\
q(v):=t_1(v,0).
\end{array}\label{condsfg}
\end{equation}
Equation (\ref{mongeampere}) immediately implies that $p$ is independent of $w$, so we write $p$ as a function of the variable $v$ alone. Furthermore, we have
\begin{equation}
q'(v)=\frac{1}{\rho_{11}(t_1(v,0),0)}.\label{gprime}
\end{equation}
Clearly, (\ref{initial}), (\ref{rho11}), (\ref{changevar}), (\ref{condsfg}), (\ref{gprime}) yield
\begin{equation}
p(0)=0,\quad q(0)=0,\quad \hbox{$q'>0$ everywhere.}\label{initialconds}
\end{equation}

In terms of $p$ and $q$, the inverse of (\ref{changevar}) is written as
\begin{equation}
\begin{array}{l}
t_1=q(v)-w\,p'(v),\\
\vspace{-0.3cm}\\
t_2=w,
\end{array}\label{inverttted}
\end{equation}
and the solution $\rho$ in the variables $v,w$ is given by
\begin{equation}
\rho(t_1(v,w),w)=vq(v)-\int_{0}^vq(\tau)d\tau+w(p(v)-vp'(v)).\label{solsparam}
\end{equation}
In particular, we see that the general smooth solution of the homogeneous Monge-Amp\`ere equation (\ref{mongeampere}) satisfying conditions (\ref{initial}), (\ref{rho11}) is parametrized by a pair of arbitrary smooth functions satisfying (\ref{initialconds}).

By \cite{MP, Poc}, the condition that $\Gamma_{\rho}$ is CR-flat is equivalent to system (\ref{condsjw}), where $J$ and $W$ are given by formulas (\ref{expljw}). Notice that, when $J$ and $W$ are equated to zero, one no longer needs to assume that substitution (\ref{substitution}) takes place, thus the right-hand sides of (\ref{expljw}) in system (\ref{condsjw}) are regarded as functions on a neighborhood of the origin in $\RR^2$.

First, we analyze the second equation in (\ref{condsjw}):

\begin{lemma}\label{wvanish}
The condition $W=0$ is equivalent to the equation $S_1=0$.
\end{lemma}

\begin{proof} From the second equation in (\ref{expljw}) it follows that the condition $W=0$ is equivalent to
\begin{equation}
4S^2S_1+S_1\left(\frac{\rho_{12}}{\rho_{11}}S_1-S_2\right)-S\left(\frac{\rho_{12}}{\rho_{11}}S_{11}-S_{12}\right)=0.\label{w0}
\end{equation}
We will now rewrite (\ref{w0}) in the variables $v$, $w$ introduced in (\ref{changevar}).

First of all, from (\ref{changevar}), (\ref{inverttted}) we compute
\begin{equation}
\begin{array}{l}
\displaystyle\rho_{11}(t_1(v,w),w)=\displaystyle\frac{1}{q'-w\,p''},\\
\vspace{-0.3cm}\\
\displaystyle\rho_{12}(t_1(v,w),w)=\displaystyle\frac{p'}{q'-w\,p''}.
\end{array}\label{secondpartials}
\end{equation}
Therefore, using (\ref{inverttted}), for any function $f(t_1,t_2)$ we see
\begin{equation}
\frac{\rho_{12}(t_1(v,w),w)}{\rho_{11}(t_1(v,w),w)}f_1(t_1(v,w),w)-f_2(t_1(v,w),w)=-\frac{\partial f(t_1(v,w),w)}{\partial w}.\label{convidentity}
\end{equation}
Next, from formulas (\ref{functions}), (\ref{changevar}), (\ref{secondpartials}) we obtain
\begin{equation}
\begin{array}{l}
\displaystyle S(t_1(v,w),w)=\frac{p''}{q'-w\,p''},\\
\vspace{-0.1cm}\\
\displaystyle S_1(t_1(v,w),w)=\frac{p'''q'-p''q''}{(q'-w\,p'')^3}.\\
\end{array}\label{ids444}
\end{equation}

It follows from (\ref{convidentity}), (\ref{ids444}) that (\ref{w0}) in the variables $v$, $w$ becomes
$$
\frac{6(p'')^2(p'''q'-p''q'')}{(q'-w\,p'')^5}=0.
$$ 
By (\ref{snonzero}) and the first equation in (\ref{ids444}) we see that $p''$ is nowhere zero, which shows that (\ref{w0}) is equivalent to the identity
\begin{equation}
p'''q'-p''q''= 0,\label{s10}
\end{equation}
that is, by the second equation in (\ref{ids444}), to the identity $S_1=0$. \end{proof}

We will now finalize the proof of Theorem \ref{flatness}. By the first formula in (\ref{expljw}) and Lemma \ref{wvanish}, we see that the first identity in system (\ref{condsjw}) becomes
\begin{equation} 
\frac{1}{3}\frac{\rho_{111}}{\rho_{11}}\left(\frac{\rho_{111}}{\rho_{11}}\right)_1-\frac{2}{27}\frac{(\rho_{111})^3}{(\rho_{11})^3}-\frac{1}{6}\left(\frac{\rho_{111}}{\rho_{11}}\right)_{11}=0,\label{neweq}
\end{equation}
and a straightforward calculation yields that (\ref{neweq}) is exactly equation (\ref{veryfinalthetasss}). \qed

\section{Equation (\ref{veryfinalthetasss}) and the classical Monge equation}\label{conclus}
\setcounter{equation}{0}

In this section, we briefly discuss how (\ref{veryfinalthetasss}), which is the Monge equation with respect to the first variable, is related to the classical single-variable Monge equation. 

Let us rewrite (\ref{veryfinalthetasss}) in coordinates $v$, $w$ introduced in (\ref{changevar}). From (\ref{changevar}) and (\ref{secondpartials}) one computes 
\begin{equation}
\hspace{0.8cm}\makebox[250pt]{$\begin{array}{l}
\displaystyle \rho_{111}(t_1(v,w),w)=-\frac{q''-w\,p'''}{(q'-w\,p'')^3},\\
\vspace{-0.1cm}\\
\displaystyle \rho^{{\rm(IV)}}(t_1(v,w),w)=-\frac{1}{(q'-w\,p'')^5}\Bigl[(q'''-w\,p^{{\rm(IV)}})(q'-w\,p'')-\\
\vspace{-0.6cm}\\
\displaystyle\hspace{9cm}3(q''-w\,p''')^2\Bigr],\\
\vspace{-0.1cm}\\
\displaystyle \rho^{{\rm(V)}}(t_1(v,w),w)=-\frac{1}{(q'-w\,p'')^7}\Bigl[\Bigl((q^{{\rm(IV)}}-w\,p^{{\rm(V)}})(q'-w\,p'')-\\
\vspace{-0.4cm}\\
\displaystyle\hspace{1cm}5(q''-w\,p''')(q'''-w\,p^{{\rm(IV)}})\Bigr)(q'-w\,p'')-\\
\vspace{-0.4cm}\\
\displaystyle\hspace{1cm}5\Bigl((q'''-w\,p^{{\rm(IV)}})(q'-w\,p'')-3(q''-w\,p''')^2\Bigr)(q''-w\,p''')\Bigr].
\end{array}$}\label{rho4and5}
\end{equation}
Plugging expressions from (\ref{secondpartials}), (\ref{rho4and5}) into (\ref{veryfinalthetasss}) and collecting coefficients at $w^k$ for $k=0,1,2,3$ in the resulting formula, we see that (\ref{veryfinalthetasss}) is equivalent to the following system of four ordinary differential equations:
\begin{equation}
\left\{
\begin{array}{l}
9p^{{\rm(V)}}(p'')^2-45p^{{\rm(IV)}}p'''p''+40(p''')^3=0,\\
\vspace{-0.1cm}\\
6p^{{\rm(V)}}p''q'+3(p'')^2q^{{\rm(IV)}}-15(p^{{\rm(IV)}}p'''q'+p^{{\rm(IV)}}p''q''+p'''p''q''')+\\
\vspace{-0.3cm}\\
\hspace{8cm}40(p''')^2q''=0,\\
\vspace{-0.3cm}\\
3p^{{\rm(V)}}(q')^2+6p''q^{{\rm(IV)}}q'-15(p^{{\rm(IV)}}q''q'+p'''q'''q'+p''q'''q'')+\\
\vspace{-0.3cm}\\
\hspace{8cm}40p'''(q'')^2=0,\\
\vspace{-0.3cm}\\
9q^{{\rm(IV)}}(q')^2-45q'''q''q'+40(q'')^3=0.\\
\end{array}
\right.\label{final1}
\end{equation}

Notice that the first entry in system (\ref{final1}) is the classical Monge equation. Also observe that all the equations in (\ref{final1}) follow from the first one if $S_1$ vanishes, i.e., if condition (\ref{s10}) is satisfied. Indeed, dividing  (\ref{s10}) by $(p'')^2$, we see that the vanishing of $S_1$ is equivalent to 
\begin{equation}
q'/p''=\hbox{const},\label{firstcur}
\end{equation}
which guarantees that each equation in (\ref{final1}) is a consequence of the first one.

In fact, the four equations in (\ref{final1}) are known to imply relation (\ref{firstcur}). This surprising result was established in \cite{I3}, thus we see that system (\ref{final1}) reduces to its first entry. In other words, (\ref{veryfinalthetasss}) in coordinates $v$, $w$ becomes the classical single-variable Monge equation
\begin{equation}
9p^{{\rm(V)}}(p'')^2-45p^{{\rm(IV)}}p'''p''+40(p''')^3=0.\label{classicalMongeeq}
\end{equation}

In the next section we will consider the more complicated case of rigid hypersurfaces and observe that it leads to a complex analogue of equation (\ref{veryfinalthetasss}).

\section{CR-curvature of rigid hypersurfaces}\label{rigid}
\setcounter{equation}{0}

Consider $\CC^3$ with standard coordinates $(z_1,z_2,z_3)$. A \emph{rigid equation} in $\CC^3$ is an equation of the form
\begin{equation}
\Re z_3=F(z_1,\overline{z}_1,z_2,\overline{z}_2),\label{rigideq}
\end{equation}
where $F$ is a smooth real-valued function defined on a domain in $\CC^2$. Such equations generalize tube ones as defined in (\ref{tubeeq}). We call any hypersurface given by a rigid equation a \emph{rigid hypersurface}. Rigid hypersurfaces have rigid CR-structures in the sense of \cite{BRT}. We only consider rigid hypersurfaces passing through the origin and will be mainly interested in their germs at 0. Thus, we assume that $F$ in (\ref{rigideq}) is defined near the origin and $F(0)=0$, with the domain of $F$ being allowed to shrink if necessary.

 Let now $M$ be a rigid hypersurface  that is uniformly Levi degenerate of rank 1. Then the complex Hessian matrix of $F$ has rank 1 at every point, hence $F$ is a solution of the \emph{complex homogeneous Monge-Amp\`ere equation}
\begin{equation}
F_{1\bar1}F_{2\bar 2}-|F_{1\bar 2}|^2=0\label{cmplxmongeampere}
\end{equation}
(here and below subscripts $1$, $\bar 1$, $2$, $\bar 2$ indicate partial derivatives with respect to $z_1$, $\bar z_1$, $z_2$, $\bar z_2$, respectively). Clearly, we have either $F_{1\bar1}(0)\ne 0$ or $F_{2\bar2}(0)\ne 0$, so by interchanging the variables $z_1, z_2$ and multiplying $z_3$ by -1 if necessary, we may additionally assume
\begin{equation}
\hbox{$F_{1\bar 1}>0$ everywhere.}\label{f11}
\end{equation}
Note that the transformation that brings equation (\ref{rigideq}) to an equation of the same form satisfying condition (\ref{f11}) establishes equivalence between the two hypersurface germs in the sense of Definition \ref{equivalencerel} given below.

Set
$$
S:=\left(\frac{F_{1\bar 2}}{F_{1\bar 1}}\right)_1\label{newdefs}
$$
(cf.~(\ref{functions})). The condition of 2-nondegeneracy is then expressed as the nonvanishing of $S$ (see \cite{MP,Poc} and cf.~\cite{I2}). Thus, assuming that $M$ is 2-nondegenerate, we have
\begin{equation}
\hbox{$S\ne 0$ everywhere.}\label{cmplxsnonzero}
\end{equation}

We now turn to Pocchiola's invariants $J$ and $W$ in the rigid case. A somewhat lengthy but straightforward computation yields that the complicated formulas of \cite{MP,Poc} simplify as shown in the following proposition, which we state without proof:

\begin{proposition}\label{complexexplinvprop}
Let $M$ be a rigid hypersurface in the class ${\mathfrak C}_{2,1}$, where the function $F$ satisfies {\rm (\ref{f11})}. Then we have 
\begin{equation}
\begin{array}{l}
\displaystyle J=\frac{5(S_1)^2}{18 S^2}\frac{F_{11\bar 1}}{F_{1\bar 1}}+\frac{1}{3}\frac{F_{11\bar 1}}{F_{1\bar 1}}\left(\frac{F_{11\bar 1}}{F_{1\bar 1}}\right)_1-\frac{S_1}{9 S}\frac{(F_{11\bar 1})^2}{(F_{1\bar 1})^2}+\frac{20(S_1)^3}{27 S^3}-\frac{5 S_1 S_{11}}{6 S^2}+\\
\vspace{-0.3cm}\\
\hspace{0.7cm}\displaystyle\frac{S_1}{6 S}\left(\frac{F_{11\bar 1}}{F_{1\bar 1}}\right)_1-\frac{S_{11}}{6 S}\frac{F_{11\bar 1}}{F_{1\bar 1}}-\frac{2}{27}\frac{(F_{11\bar 1})^3}{(F_{1\bar 1})^3}-\frac{1}{6}\left(\frac{F_{11\bar1}}{F_{1\bar 1}}\right)_{11}+\frac{S_{111}}{S_1},\\
\vspace{-0.1cm}\\
\displaystyle W=\frac{2\bar S_1}{3\bar S}+\frac{2S_1}{3S}+\frac{\bar S_{\bar 1}}{3\bar S^3}\left(\frac{F_{2\bar 1}}{F_{1\bar 1}}\bar S_1-\bar S_2\right)-\frac{1}{3\bar S^2}\left(\frac{F_{2\bar 1}}{F_{1\bar 1}}\bar S_{1\bar 1}-\bar S_{2\bar 1}\right).
\end{array}\label{cmplxexpljw}
\end{equation}
\end{proposition}

Just as in the tube case, we are interested in classifying the germs of CR-flat rigid hypersurfaces in ${\mathfrak C}_{2,1}$. Certainly, in order to speak about a classification, one must decide what germs are to be called equivalent. In contrast to the tubular case, where we utilized affine equivalence, there is no widely accepted notion of equivalence in the rigid situation, and we introduce one as follows: 

\begin{definition}\label{equivalencerel}
Two germs of rigid hypersurfaces at 0 are called \emph{rigidly equivalent} if there exists a map of the form
\begin{equation}
(z_1,z_2,z_3)\mapsto (f(z_1,z_2),g(z_1,z_2),a z_3+h(z_1,z_2)),\,\,a\in \RR^*,\label{rigideqmaps}
\end{equation}
nondegenerate at the origin, where $f,g,h$ are functions holomorphic near 0 with $f(0)=g(0)=h(0)=0$, that maps one hypersurface germ into the other. 
\end{definition}

One should be warned that two affinely equivalent tube hypersurface germs defined by tube equations are not necessarily rigidly equivalent. 

Finding the germs of the graphs of all solutions of system (\ref{condsjw}) in the rigid case up to rigid equivalence is apparently very hard. Below, we will only make some initial steps towards this goal. Specifically, we will discuss solutions having the property
\begin{equation}
S_1=0,\quad S_{\bar 1}=0.\label{s1111}
\end{equation}
Our motivation for introducing conditions (\ref{s1111}) comes from the tube case, where these conditions are equivalent to the equation $W=0$ (see Lemma \ref{wvanish}). At this stage, we do not know whether the same holds in the rigid case as well, but it is clear from (\ref{cmplxexpljw}) that (\ref{s1111}) implies $W=0$. 

Furthermore, conditions (\ref{s1111}) lead to the following simplified expression for $J$:
\begin{equation}
\displaystyle J=\frac{1}{3}\frac{F_{11\bar 1}}{F_{1\bar 1}}\left(\frac{F_{11\bar 1}}{F_{1\bar 1}}\right)_1-\frac{2}{27}\frac{(F_{11\bar 1})^3}{(F_{1\bar 1})^3}-\frac{1}{6}\left(\frac{F_{11\bar1}}{F_{1\bar 1}}\right)_{11}.\label{reducedj}
\end{equation}
Formula (\ref{reducedj}) yields that under assumption (\ref{s1111}) the equation $J=0$ is equivalent to
\begin{equation}
9 F_{1111\bar 1}(F_{1\bar 1})^2-45 F_{111\bar 1}F_{11\bar 1} F_{1\bar 1}+40 (F_{11\bar 1})^3=0,\label{mplxmongeeq}
\end{equation}
which looks remarkably similar to the Monge equation with respect to the first variable (\ref{veryfinalthetasss}). We call (\ref{mplxmongeeq}) the \emph{complex Monge equation} with respect to $z_1$.

\begin{remark}\label{tubevsrigid}
Recall that our proof of the equivalence of conditions (\ref{s1111}) and the equation $W=0$ in the tube case given in Lemma \ref{wvanish} relied on representation (\ref{solsparam}) of the solutions of the real homogeneous Monge-Amp\`ere equation (\ref{mongeampere}), which was based on change of variables (\ref{changevar}). This representation was also the key point of our proof in \cite{I3} of the fact that for tube hypersurfaces equation (\ref{mplxmongeeq}) implies conditions (\ref{s1111}). As there is no analogue of (\ref{changevar}) for the complex homogeneous Monge-Amp\`ere equation (\ref{cmplxmongeampere}), the proofs that worked in the tube case do not immediately generalize to the rigid one. 
\end{remark}

Thus, we arrive at a natural class of CR-flat rigid hypersurfaces in ${\mathfrak C}_{2,1}$ described by the system of partial differential equations
$$
\left\{\begin{array}{l}
\hbox{the complex Monge equation w.r.t. $z_1$ (\ref{mplxmongeeq})},\\
\vspace{-0.3cm}\\
\hbox{the complex Monge-Amp\`ere equation (\ref{cmplxmongeampere})},\\
\vspace{-0.3cm}\\
\hbox{equations (\ref{s1111})},
\end{array}\right.
$$
where conditions (\ref{f11}) and (\ref{cmplxsnonzero}) are satisfied. This system may be viewed as a complex analogue of (\ref{twoeeqs}).

We will now prove that (\ref{mplxmongeeq}) can be integrated three times with respect to $z_1$. An analogous fact holds for the Monge equations (\ref{veryfinalthetasss}) and (\ref{classicalMongeeq}) (see, e.g., \cite{I2, I3}). 

\begin{proposition}
A function $F$ satisfying {\rm (\ref{f11})} is a solution of {\rm (\ref{mplxmongeeq})} if and only if
\begin{equation}
\begin{array}{l}
\displaystyle\frac{1}{(F_{1\bar 1})^{\frac{2}{3}}}=f(z_2,\bar z_2)|z_1|^4+g(z_2,\bar z_2)z_1^2\bar z_1+\overline{g(z_2,\bar z_2)}z_1\bar z_1^2+h(z_2,\bar z_2)|z_1|^2+\\
\vspace{-0.3cm}\\
\hspace{1.5cm}\displaystyle p(z_2,\bar z_2)z_1^2+\overline{p(z_2,\bar z_2)}\bar z_1^2+q(z_2,\bar z_2)z_1+\overline{q(z_2,\bar z_2)}\bar z_1+v(z_2,\bar z_2),
\end{array}\label{intthreetimes}
\end{equation}
where $f,g,h,p,q,v$ are smooth functions, with $f, h, v$ being real-valued.
\end{proposition}

\begin{proof} Notice that
$$
\begin{array}{l}
\displaystyle\left(\frac{9F_{111\bar 1}}{(F_{1\bar 1})^{\frac{5}{3}}}-\frac{15(F_{11\bar 1})^2}{(F_{1\bar 1})^{\frac{8}{3}}}\right)_1=\frac{9 F_{1111\bar 1}(F_{1\bar 1})^2-45 F_{111\bar 1}F_{11\bar 1} F_{1\bar 1}+40 (F_{11\bar 1})^3}{(F_{1\bar 1})^{\frac{11}{3}}},\\
\vspace{-0.1cm}\\
\displaystyle\left(\frac{9 F_{11\bar 1}}{(F_{1\bar 1})^{\frac{5}{3}}}\right)_1=\frac{9F_{111\bar 1}}{(F_{1\bar 1})^{\frac{5}{3}}}-\frac{15(F_{11\bar 1})^2}{(F_{1\bar 1})^{\frac{8}{3}}},\\
\vspace{-0.1cm}\\
\displaystyle\left(-\frac{27}{2 (F_{1\bar 1})^{\frac{2}{3}}}\right)_1=\frac{9F_{11\bar 1}}{(F_{1\bar 1})^{\frac{5}{3}}}.
\end{array}
$$
Therefore $F$ is a solution of (\ref{mplxmongeeq}) if and only if 
\begin{equation}
\frac{1}{(F_{1\bar 1})^{\frac{2}{3}}}=a(\bar z_1, z_2,\bar z_2)z_1^2+b(\bar z_1,z_2,\bar z_2)z_1+c(\bar z_1,z_2,\bar z_2),\label{intthreetimesprelim}
\end{equation}
where $a,b,c$ are smooth functions. Taking into account that $(F_{1\bar 1})^{\frac{2}{3}}$ is real-valued, we see that if $F$ satisfies (\ref{intthreetimesprelim}) then each of $a$, $b$, $c$ is a polynomial in $\bar z_1$ of degree at most 2, thus (\ref{intthreetimesprelim}) is equivalent to representation (\ref{intthreetimes}).\end{proof}

We will now discuss the simplest possible situation, specifically, the case when in formula (\ref{intthreetimes}) one has $f=g=h=p=q=0$. In other words, let
$$
F_{1\bar 1}=r(z_2,\bar z_2)
$$  
or, equivalently,
$$
F=r(z_2,\bar z_2)|z_1|^2+s(z_1,z_2,\bar z_2)+\overline{s(z_1,z_2,\bar z_2)}\label{specform1}
$$
for some smooth functions $r$ and $s$. Finding the germs of all suitable functions $F$ even in this constrained situation is hard, and we let $s$ be of the form
\begin{equation}
s(z_1,z_2,\bar z_2)=t(z_2,\bar z_2)z_1^2\label{anotherc}
\end{equation}
for a smooth function $t$, so we have
\begin{equation}
F=r(z_2,\bar z_2)|z_1|^2+t(z_2,\bar z_2)z_1^2+\overline{t(z_2,\bar z_2)}\bar z_1^2.\label{veryspecform}
\end{equation}

Condition (\ref{anotherc}) is motivated by the following well-known example of the germ of a CR-flat rigid hypersurface with the graphing function $F$ as in (\ref{veryspecform}), which was given in \cite[Proposition 4.16]{FK} (see also \cite{GM}):
\begin{equation}
\Re z_3=\frac{|z_1|^2}{1-|z_2|^2}+\frac{\bar z_2}{2(1-|z_2|^2)}z_1^2+\frac{z_2}{2(1-|z_2|^2)}\bar z_1^2,\label{bestknownexample}
\end{equation}
where we have
\begin{equation}
r=\frac{1}{1-|z_2|^2},\quad t=\frac{\bar z_2}{2(1-|z_2|^2)}.\label{specialrt}
\end{equation}
Theorem \ref{rigidclassif} below contains more nontrivial examples of this kind.

For a rigid hypersurface with the graphing function $F$ of the form (\ref{veryspecform}) the totality of the Monge-Amp\`ere equation (\ref{cmplxmongeampere}) and conditions (\ref{f11}), (\ref{cmplxsnonzero}) is easily seen to be equivalent to the following set of relations:
\begin{equation}
\begin{array}{l}
r r_{2\bar 2}-|r_2|^2-4|t_{\bar 2}|^2=0,\\
\vspace{-0.1cm}\\
r t_{2\bar 2}-2r_2t_{\bar 2}=0,\\
\vspace{-0.1cm}\\
\hbox{$r>0$ everywhere,}\\
\vspace{-0.1cm}\\
\hbox{$t_{\bar 2}\ne 0$ everywhere.}
\end{array}\label{mainsystemm}
\end{equation}   
We will look for a pair of functions $r,t$ satisfying (\ref{mainsystemm}) assuming that
\begin{equation}
t_{\bar 2}=\frac{r^2}{2}.\label{reltr}
\end{equation}
Just like condition (\ref{anotherc}), this relation is motivated by example (\ref{bestknownexample}) as the functions $r,t$ from (\ref{specialrt}) satisfy it. Clearly, under constraint (\ref{reltr}) the second equation in (\ref{mainsystemm}) trivially holds and the last relation follows from the third one. Thus, (\ref{mainsystemm}) turns into
\begin{equation}
\begin{array}{l}
r r_{2\bar 2}-|r_2|^2-r^4=0,\\
\vspace{-0.1cm}\\
\hbox{$r>0$ everywhere.}\\
\end{array}\label{newshortsys}
\end{equation}

Set
$f:=\ln r$. By (\ref{newshortsys}) the function $f$ satisfies
$$
f_{2\bar 2}=e^{2f},
$$
or
\begin{equation}
\Delta f=4 e^{2f}.\label{laplace}
\end{equation}
Equation (\ref{laplace}) has many solutions, which are related to conformal metrics of constant negative curvature (see, e.g., \cite{KR}). By formulas (\ref{veryspecform}), (\ref{reltr}) every solution yields an example of the germ of a CR-flat rigid hypersurface in the class ${\mathfrak C}_{2,1}$. Thus, we see that even under constraints (\ref{s1111}), (\ref{veryspecform}), (\ref{reltr}) there are a large number of examples, and it is not clear whether they all can be found up to rigid equivalence.

We will now introduce a further assumption on the function $F$ by letting
\begin{equation}
r(z_2,\bar z_2)=R(z_2+\bar z_2)\label{tubeassumr}
\end{equation}
for some smooth positive function of one variable $R(x)$ defined near the origin. Just like (\ref{s1111}), this constraint is motivated by the tube case. As shown below, the introduction of (\ref{tubeassumr}) makes it possible to come up with an explicit partial classification.

\begin{theorem}\label{rigidclassif}
The germ of a rigid hypersurface in the class ${\mathfrak C}_{2,1}$ with the graphing function $F$ satisfying {\rm (\ref{s1111}), (\ref{veryspecform}), (\ref{reltr}), (\ref{tubeassumr})} is rigidly equivalent to the germ of one of the following:
\begin{itemize}

\item[{\rm (i)}] the tube hypersurface given by an equation of the form
$$
\Re z_3=\frac{(z_1+\bar z_1)^2}{z_2+\bar z_2+D},\quad D>0;
$$
\vspace{0.1cm}

\item[{\rm (ii)}] the hypersurface given by an equation of the form
$$
\Re z_3=\frac{z_1^2+2\sqrt{D}|z_1|^2|z_2+1|^2+\bar z_1^2}{1-D|z_2+1|^4}, \quad 0<D<1;
$$
\vspace{0.1cm}

\item[{\rm (iii)}] the hypersurface given by an equation of the form
$$
\Re z_3=\frac{i\left(e^{iD}(z_2+1)^2-e^{-iD}(\bar z_2+1)^2\right)(z_1^2+\bar z_1^2)-4|z_1|^2|z_2+1|^2}{e^{iD}(z_2+1)^2+e^{-iD}(\bar z_2+1)^2},\quad 0<D<\pi/2.
$$
\end{itemize}
\end{theorem}

\begin{proof}

Set $g:=\ln R$. By (\ref{newshortsys}) the function $g$ satisfies
\begin{equation}
g''=e^{2g}.\label{laplacetube}
\end{equation}
Multiplying both sides of (\ref{laplacetube}) by $g'$ and integrating, we obtain
$$
(g')^2=e^{2g}+C
$$
for some $C\in\RR$, hence 
$$
g'=\sigma\sqrt{e^{2g}+C},\label{eqtosolve}
$$
with $\sigma=\pm 1$. We will now consider three cases.
\vspace{0.1cm}

{\bf Case 1.} Let $C=0$. Then
$$
R(x)=\frac{1}{-\sigma x+D},
$$
where $D$ is a positive constant, and (\ref{reltr}) yields
$$
t(z_2,\bar z_2)=\frac{\sigma}{2(-\sigma(z_2+\bar z_2)+D)}+u(z_2),
$$
where $u$ is a holomorphic function. By (\ref{veryspecform}) we then have
$$
F=\frac{\sigma}{2(-\sigma(z_2+\bar z_2)+D)}(z_1^2+2\sigma|z_1|^2+\bar z_1^2)+2\Re(z_1^2u(z_2)). 
$$
The germ of the rigid hypersurface with the above graphing function is easily seen to be rigidly equivalent to the germ of the tube hypersurface
$$
\Re z_3=\frac{(z_1+\bar z_1)^2}{z_2+\bar z_2+D},
$$
which is the possibility given by part (i) of the theorem.
\vspace{0.1cm}

{\bf Case 2.} Let $C>0$. We then have
$$
R(x)=\frac{2\sqrt{CD}e^{\sigma\sqrt{C}x}}{1-De^{2{\sigma\sqrt{C}x}}},
$$
where $0<D<1$. From (\ref{reltr}) we thus see
$$
t(z_2,\bar z_2)=\frac{\sigma\sqrt{C}}{1-De^{2{\sigma\sqrt{C}(z_2+\bar z_2)}}}+u(z_2),
$$
where $u$ is a holomorphic function. By (\ref{veryspecform}) it follows that
$$
F=\frac{\sigma\sqrt{C}}{1-De^{2{\sigma\sqrt{C}(z_2+\bar z_2)}}}(z_1^2+2\sigma\sqrt{D} e^{\sigma\sqrt{C}(z_2+\bar z_2)}|z_1|^2+\bar z_1^2)+2\Re(z_1^2u(z_2)).
$$
The germ of the rigid hypersurface with this graphing function is rigidly equivalent to the germ of the hypersurface
$$
\Re z_3=\frac{z_1^2+2\sqrt{D}e^{z_2+\bar z_2}|z_1|^2+\bar z_1^2}{1-De^{2(z_2+\bar z_2)}}.
$$
Now, to eliminate the exponential function from the equation, we replace $e^{z_2}$ by $z_2+1$ and obtain
$$
\Re z_3=\frac{z_1^2+2\sqrt{D}|z_1|^2|z_2+1|^2+\bar z_1^2}{1-D|z_2+1|^4},
$$
which is the possibility given by part (ii).
\vspace{0.1cm}

{\bf Case 3.} Let $C<0$. Then
$$
R(x)=\frac{\sqrt{-C}}{\cos(\sigma\sqrt{-C}x+D)},
$$
where $0<D<\pi/2$, and (\ref{reltr}) implies
$$
t(z_2,\bar z_2)=\frac{\sigma\sqrt{-C}}{2}\tan(\sigma\sqrt{-C}(z_2+\bar z_2)+D)+u(z_2),
$$
where $u$ is a holomorphic function. By (\ref{veryspecform}) it follows that
$$
\begin{array}{l}
\displaystyle F=\frac{\sigma\sqrt{-C}}{2\cos(\sigma\sqrt{-C}(z_2+\bar z_2)+D)}\times\\
\vspace{-0.1cm}\\
\hspace{2cm}\displaystyle\left(\sin(\sigma\sqrt{-C}(z_2+\bar z_2)+D)(z_1^2+\bar z_1^2)+2\sigma|z_1|^2\right)+2\Re(z_1^2u(z_2)). 
\end{array}
$$
The germ of the rigid hypersurface with the above graphing function is rigidly equivalent to the germ of the hypersurface
$$
\Re z_3=\frac{\sin(z_2+\bar z_2+D)(z_1^2+\bar z_1^2)+2|z_1|^2}{\cos(z_2+\bar z_2+D)}.
$$
Finally, to eliminate the exponential function from the equation, we replace $e^{iz_2}$ by $z_2+1$ and obtain
$$
\Re z_3=\frac{i\left(e^{iD}(z_2+1)^2-e^{-iD}(\bar z_2+1)^2\right)(z_1^2+\bar z_1^2)-4|z_1|^2|z_2+1|^2}{e^{iD}(z_2+1)^2+e^{-iD}(\bar z_2+1)^2},
$$
which is the possibility given by part (iii). \end{proof}

\begin{remark}\label{remrigideq}
Despite the fact that constraints (\ref{s1111}), (\ref{tubeassumr}) are of \lq\lq tube type\rq\rq, the equations that appear in parts (ii) and (iii) of Theorem \ref{rigidclassif} are \emph{not} tube equations and it is not clear whether they can be brought to tube form by transformations of the kind specified in (\ref{rigideqmaps}).
\end{remark} 

Theorem \ref{rigidclassif} immediately implies:

\begin{corollary}\label{rational}
The germ of a rigid hypersurface in the class ${\mathfrak C}_{2,1}$ with the graphing function $F$ satisfying {\rm (\ref{s1111}), (\ref{veryspecform}), (\ref{reltr}),(\ref{tubeassumr})} is rigidly equivalent to the germ of a hypersurface with rational graphing function.
\end{corollary}

\begin{remark}\label{remrational}
It is probably not realistic to expect that one can find an explicit classification if the simplifying assumptions (\ref{s1111}), (\ref{veryspecform}), (\ref{reltr}), (\ref{tubeassumr}) of Theorem \ref{rigidclassif} are completely dropped. On the other hand, the rationality result stated in Corollary \ref{rational} might still hold, and a very interesting open question is whether it can be established in full generality. 
\end{remark}

In conclusion, we will investigate another natural situation when $r$ is expressed via a function of one variable. Namely, instead of (\ref{tubeassumr}) let us assume that 
\begin{equation}
r(z_2,\bar z_2)=R(|z_2|^2)\label{reinhaassumr}
\end{equation}
for some smooth positive function $R(x)$ defined near the origin. This constraint is motivated by example (\ref{bestknownexample}). We have the following result:

\begin{proposition}\label{rigidclassifreinh}
The germ of a rigid hypersurface in the class ${\mathfrak C}_{2,1}$ with the graphing function $F$ satisfying {\rm (\ref{s1111}), (\ref{veryspecform}), (\ref{reltr}), (\ref{reinhaassumr})} is rigidly equivalent to the germ of the hypersurface given by equation {\rm (\ref{bestknownexample})}.
\end{proposition}

\begin{proof}
Set $g:=\ln R$. By (\ref{newshortsys}) the function $g$ satisfies
\begin{equation}
g''x+g'=e^{2g},\label{eqlaplacerig}
\end{equation}
or, equivalently,
\begin{equation}
(g'x)'=e^{2g}.\label{eqlaplacerig1}
\end{equation}
Multiplying both sides of (\ref{eqlaplacerig1}) by $g'x$ we obtain
\begin{equation}
((g'x)^2)'=(e^{2g})'x.\label{eqlaplacerig2}
\end{equation}
Adding up (\ref{eqlaplacerig1}), (\ref{eqlaplacerig2}) and integrating, we see
$$
(g'x)^2+g'x=e^{2g}x+C
$$
for some $C\in\RR$. By setting $x=0$ in the above equation we observe that $C=0$, which yields
\begin{equation}
(g')^2x+g'=e^{2g}.\label{eqlaplacerig3} 
\end{equation}
Now, the comparison of (\ref{eqlaplacerig}) and (\ref{eqlaplacerig3}) implies
$$
g''=(g')^2,
$$
which leads to
\begin{equation}
g=\alpha-\ln|x+D|\label{eqlaplacerig4}
\end{equation}
for some $\alpha,D\in\RR$. By plugging (\ref{eqlaplacerig4}) into the original equation (\ref{eqlaplacerig}), we finally obtain
$$
g(x)=\beta-\ln(1-e^{2\beta}x)
$$
for some $\beta\in\RR$.

Then (\ref{reltr}) yields
$$
t(z_2,\bar z_2)=\frac{e^{2\beta}\bar z_2}{2(1-e^{2\beta}|z_2|^2)}+u(z_2),
$$
where $u$ is a holomorphic function. By (\ref{veryspecform}) it follows that
$$
F=\frac{e^{\beta}}{1-e^{2\beta}|z_2|^2}|z_1|^2+\frac{e^{2\beta}\bar z_2}{2(1-e^{2\beta}|z_2|^2)}z_1^2+\frac{e^{2\beta}z_2}{2(1-e^{2\beta}|z_2|^2)}\bar z_1^2+2\Re(z_1^2u(z_2)).
$$
The germ of the rigid hypersurface with this graphing function is easily seen to be rigidly equivalent to the germ of the hypersurface defined by equation (\ref{bestknownexample}). \end{proof}

We thus see that constraints (\ref{s1111}), (\ref{veryspecform}), (\ref{reltr}), (\ref{reinhaassumr}) do not lead to any new examples.

\end{document}